\documentclass[12pt]{article}

\usepackage[margin=1in]{geometry}
\usepackage{amsthm}
\usepackage{amssymb}
\usepackage{amsmath}
\usepackage{graphicx}
\usepackage{mathdots}
\usepackage{mathtools}
\usepackage{xurl}
\usepackage{color}
\usepackage{framed}

\newtheorem{theorem}{Theorem}

\newtheorem{lemma}[theorem]{Lemma}

\newtheorem{conjecture}[theorem]{Conjecture}

\newcommand{\eps}{\varepsilon}
\newcommand{\EE}{\mathbb{E}}
\newcommand{\PP}{\mathbb{P}}
\newcommand{\RR}{\mathbb{R}}

\newcommand{\NNN}{\mathcal{N}}
\newcommand{\SSS}{\mathbb{S}}
\newcommand{\cov}{\mathrm{Cov}}
\newcommand{\Tr}{\mathrm{Tr}}
\newcommand{\polylog}{\mathrm{polylog}}

\theoremstyle{definition}

\title{Matrix Discrepancy for Representations of Finite Groups}

\author{Afonso~S.~Bandeira\footnote{\texttt{bandeira@math.ethz.ch}, Department of Mathematics, ETH Z\"{u}rich.} \qquad Helmut B\"olcskei\footnote{\texttt{hboelcskei@ethz.ch}, Chair for Mathematical Information Science, ETH Z\"{u}rich.}}
\date{}

\begin{document}
\maketitle

\begin{abstract}
We prove the group version of the Matrix Spencer conjecture. For every finite group $G$, there exist signs $\eps\in\{\pm1\}^G$ such that $$\left\| \sum_{g\in G} \eps_g\rho(g) \right\|\leq C\, \sqrt{|G|},$$ where $\rho$ is the left regular representation of $G$ and $C$ is a universal constant.
This conjecture was posed in~\cite{Bandeira:Cayley24}, which settled it for simple groups; we establish it for all finite groups, combining the Peter--Weyl decomposition with the intrinsic-freeness inequalities of~\cite{BandeiraBvH:21} in an iterated partial-coloring argument.
\end{abstract}

\section{Introduction}

The Matrix Spencer conjecture is a far-reaching generalization of Spencer's celebrated ``Six standard deviations suffice'' theorem~\cite{Spencer:85,Gluskin1989}. 

\begin{conjecture}\label{conj:matrixspencer}\cite{Zouzias:12,Meka:14,10L42OP-Bandeira}
 There exists a universal constant $C$ such that for any integer $n$ and any set of self-adjoint\footnote{Note that the self-adjoint requirement is cosmetic, as the Hermitian dilation $A\mapsto\left[\begin{array}{cc} 0 & A \\ A^\ast & 0 \end{array}\right]$ reduces the general case to the self-adjoint one.} $n\times n$ matrices $A_1,\dots,A_n$ each with spectral norm at most one, 
 \[
\min_{\eps\in\{\pm1\}^n} \left\| \sum_{k=1}^n \eps_k A_k \right\| \leq C\sqrt{n}.
 \]
\end{conjecture}

The Noncommutative Khintchine inequality of Lust-Piquard and Pisier~\cite{LustPiquardP:91,Pisier:03} shows that a uniformly random draw of $\eps$ yields the desired bound, in expectation, up to a logarithmic factor. When the matrices commute, Conjecture~\ref{conj:matrixspencer} reduces to Spencer's Theorem~\cite{Spencer:85}, where a uniformly random choice does not suffice and the argument is entropic rather than based on concentration of measure. At the opposite extreme, when the matrices behave very noncommutatively (more precisely, freely), the intrinsic-freeness matrix concentration inequalities of~\cite{BandeiraBvH:21,Bandeira2026ICM} show that a uniformly random choice of signs suffices. Notable partial progress includes establishing the conjecture for block matrices with small blocks~\cite{LevyRR:17} and for low-rank matrices (at first up to rank $\sqrt{n}$ by~\cite{HopkinsRS:22,DadushJR:22} and later up to rank $n/\polylog(n)$ by~\cite{BansalJM:22}), as well as for rank-one matrices~\cite{KyngLuhSong:20}, established via the interlacing-polynomial method of Marcus, Spielman, and Srivastava~\cite{MSS:15} underlying the resolution of the Kadison--Singer problem. Resolving the conjecture in the general case remains an open problem.

Motivated by the dichotomy of behaviors for the commutative and ``very noncommutative'' cases, \cite{Bandeira:Cayley24} proposed a group version of this conjecture in the hope that noncommutativity could be better understood in such an algebraically structured context. In this setting, $G$ is a finite group (with $n=|G|$) and $A_1,\dots,A_n$ are the matrices of the regular representation of $G$. The same reference resolved this question for simple groups and left the general finite-group case as a conjecture (see also~\cite[Conjecture 2]{Randomstrasse101problems2024}), which we prove in this paper.\footnote{After this paper was posted, Akbas and Sra~\cite{AkbasSra:26} independently obtained, by a different argument based on multiscale entropy nets rather than intrinsic freeness, a more general algebraic Matrix Spencer theorem for families contained in a finite-dimensional $C^\ast$-algebra of dimension $O(n)$, of which the group case is a special case.}

\begin{theorem}\label{thm:main}
There exists a universal constant $C$ such that, for any finite group $G$,
$$\min_{\eps\in\{\pm1\}^{G}}\left\| \sum_{g\in G} \eps_g\rho(g) \right\|\leq C\, \sqrt{|G|},$$
where $\rho$ is the left regular representation of $G$.
\end{theorem}

In fact, $\rho$ can be any unitary representation of $G$, since the left regular representation contains every irreducible representation of $G$, and any unitary representation can be decomposed into irreducibles. Our proof proceeds by simply controlling all irreducible representations simultaneously and does not explicitly use that $\rho$ is the regular representation of $G$.

The proof of Theorem~\ref{thm:main} proceeds via rounds of partial coloring following Gluskin's method. The authors of~\cite{Bandeira:Cayley24} (who overlap with those of the present paper) wrote ``The authors were able to adapt Gluskin's volume argument 
[...] 
to produce a partial coloring, but we were not able to iterate the procedure.'' That difficulty is overcome here by exploiting intrinsic freeness~\cite{BandeiraBvH:21} to iterate the procedure, even though the underlying sets will in general no longer be groups after the first iteration.

\medskip

\textbf{Outline and Notation:} Section~\ref{sec:gaussianlowerboun} contains the crux of the proof, Lemmata~\ref{lemma:gaussianmeasureLBforC} and \ref{lemma:matrixconcentrationestimateforS}, while Section~\ref{sec:roundspartialcoloring} uses these lemmata to prove Theorem~\ref{thm:main}. We use lowercase $c_x$ to denote universal constants that can be replaced by smaller constants, and $C_x$ for universal constants that can be taken larger. We make no effort to keep track of the values of universal constants.

\section{A Gaussian measure lower bound over subsets}\label{sec:gaussianlowerboun}

The goal of this section is to prove the following lemma, which is the crux of the argument and the stage at which the group structure is exploited.

\begin{lemma}\label{lemma:gaussianmeasureLBforC}
    For every $c_2>0$, there exists $C_1$ (not depending on $G$) such that,
    for all non-empty $S\subseteq G$,
    \[
\PP\left\{ \left\|\sum_{g\in S}x_g\rho(g)\right\| \leq C_1\sqrt{|S|\log \left(e\frac{|G|}{|S|}\right)}  \right\} \geq \exp(-c_2|S|),
    \]
    where $x_g$ are i.i.d. (real) $\NNN(0,1)$ random variables.
\end{lemma}

The idea is to use the Peter--Weyl theorem to decompose $\rho$ into irreducible representations. Before proving Lemma~\ref{lemma:gaussianmeasureLBforC}, we first establish a matrix concentration result for the spectral norm of the blocks arising from this decomposition.


\begin{lemma}\label{lemma:matrixconcentrationestimateforS}
    Let $\lambda_\pi$ denote an irreducible representation of $G$ of dimension $d_\pi$. 
For any $C_3>0$ there exists $C_1$ (not depending on $G$ or $\lambda_\pi$) such that, for any non-empty subset $S\subseteq G$, we have the following estimate: 
    \begin{equation}
\PP\left\{ \left\|\sum_{g\in S}x_g\lambda_\pi(g)\right\| > C_1 \sqrt{|S|\log \left(e\frac{|G|}{|S|}\right)} \right\} \leq \exp\left[-\, C_3  \log \left(e\frac{|G|}{|S|}\right)\right],
    \end{equation}
    where the $x_g$ are i.i.d. (real) $\NNN(0,1)$ random variables.
\end{lemma}

\begin{proof}[Proof of Lemma~\ref{lemma:matrixconcentrationestimateforS}]

Consider the Gaussian random matrix $X_{\pi,S} = \sum_{g\in S}x_g\lambda_\pi(g)$. The relevant quantities are the matrix variance $\sigma(X_{\pi,S})$, the weak variance $\sigma_\ast(X_{\pi,S})$, and the so-called $v$-parameter $v(X_{\pi,S})$, given by

\begin{align}
\sigma(X_{\pi,S})^2 &:= \max\left\{
\left\| \EE X_{\pi,S}X_{\pi,S}^\ast \right\|,
\left\| \EE X_{\pi,S}^\ast X_{\pi,S} \right\|
\right\} 
=
|S|, \label{eq:Computingsigma}
\\
\sigma_\ast(X_{\pi,S})^2 &:= \max_{u,v\in\SSS^{d_\pi-1}} \EE |u^\ast X_{\pi,S} v|^2 \leq |S|,\label{eq:Computingsigmastar}
\\
v(X_{\pi,S})^2 &:= \left\| \cov(X_{\pi,S})\right\| \leq \left\| \cov(X_{\pi,G})\right\| = \frac{|G|}{d_\pi}. \label{eq:Computingvparameter}
\end{align}

A direct computation using $\lambda_\pi(g)\lambda_\pi(g)^\ast = \lambda_\pi(g)^\ast\lambda_\pi(g)=I$ establishes~\eqref{eq:Computingsigma}. Since $\|\lambda_\pi(g)\|=1$, we have $\sigma_\ast(X_{\pi,S})^2\leq |S|$, which is~\eqref{eq:Computingsigmastar}. It remains to prove~\eqref{eq:Computingvparameter}, which we use in the high-dimensional case below. 
To this end, note that
\[
\max_{\|B\|_F=1}\sum_{g\in G} |\Tr(B^\ast  \lambda_\pi(g) )|^2 = \sum_{g\in G}\left|\sum_{i,j=1}^{d_\pi} \overline{B_{ij}}\lambda_\pi(g)_{ij}\right|^2 =  \sum_{i,j,k,\ell=1}^{d_\pi} \overline{B_{ij}}B_{k\ell}\sum_{g\in G}\lambda_\pi(g)_{ij}\overline{\lambda_\pi(g)_{k\ell}}. 
\]
Schur orthogonality yields $\sum_{g\in G}\lambda_\pi(g)_{ij}\overline{\lambda_\pi(g)_{k\ell}}=\delta_{ik}\delta_{j\ell} \frac{|G|}{d_\pi}$, which establishes~\eqref{eq:Computingvparameter}.\footnote{There are two other ways of viewing this computation: (i)~Schur orthogonality
implies that $\{\lambda_\pi(g)\}_{g\in G}$ is a tight frame with respect to the
Hilbert--Schmidt inner product; or (ii)~we are, in essence, computing the $v$-parameter for a
Wigner matrix of size $d_\pi\times d_\pi$ with entrywise variance
$\frac{|G|}{d_\pi}$ (see Lemma~7 in~\cite{Bandeira:Cayley24}).
}
Lipschitz Gaussian concentration (e.g.~\cite[Theorem 8.9]{MDSbook}) gives 
\begin{equation*}
\PP\left\{
\|X_{\pi,S}\| > \EE\|X_{\pi,S}\| + t \sigma_\ast(X_{\pi,S}) 
\right\}
\leq \exp\left( \frac{-t^2}2 \right),
\end{equation*}
for any $t\geq 0$. Recall that $\sigma_\ast(X_{\pi,S}) \leq \sqrt{|S|}$. Thus, setting $t = \sqrt{2C_3  \log\left(e\frac{|G|}{|S|}\right)}$, we have

\begin{equation}\label{eq:tailboundafterGC}
\PP\left\{
\|X_{\pi,S}\| > \EE\|X_{\pi,S}\| + \sqrt{2C_3  \log\left(e\frac{|G|}{|S|}\right)} \sqrt{|S|} 
\right\}
\leq \exp\left[-C_3  \log \left(e\frac{|G|}{|S|}\right)\right].
\end{equation}

It remains only to establish the upper bound 
$\EE\|X_{\pi,S}\|
\leq C_4\sqrt{|S|\log\left(e\frac{|G|}{|S|}\right)}$
for some constant $C_4$. Here, we treat low- and high-dimensional representations differently. For low-dimensional representations, the Noncommutative Khintchine inequality~\cite[Theorem 9.1]{MDSbook} gives
\begin{equation}
\EE
\|X_{\pi,S}\|  
\leq 
\sigma(X_{\pi,S}) \sqrt{2\lceil\log(d_\pi)\rceil+1} = \sqrt{(2\lceil\log(d_\pi)\rceil+1)|S|},
\end{equation}
which provides the desired upper bound as long as $d_\pi \leq \left(e\frac{|G|}{|S|} \right)^{C_5}$ (and $C_5=4$ will suffice).

For high-dimensional representations, we use the intrinsic-freeness inequalities in \cite{BandeiraBvH:21}. Namely, when $d_\pi > \left(e\frac{|G|}{|S|} \right)^{4}$, \cite[Theorem 1.2]{BandeiraBvH:21} gives

\begin{align}
\EE\|X_{\pi,S}\|
&\leq 2\sigma(X_{\pi,S}) + C_6 \log(d_\pi)^{\frac34}\sqrt{v(X_{\pi,S})\sigma(X_{\pi,S})} \nonumber\\
&= \sqrt{|S|}\left(2 + C_6 \log(d_\pi)^{\frac34}\left(\frac{|G|}{d_\pi|S|}\right)^{\frac14}\right) \nonumber\\
&= \sqrt{|S|}\left(2 + C_6 \left(\frac{\log(d_\pi)^3}{d_\pi}\,\frac{|G|}{|S|}\right)^{\frac14}\right) \nonumber\\
&\leq \sqrt{|S|}\left(2 + C_6 \left(\frac{4^3\log\left(e\frac{|G|}{|S|}\right)^3}{\left(e|G|/|S|\right)^{4}}\,\frac{|G|}{|S|}\right)^{\frac14}\right) \nonumber\\
&\leq C_7 \sqrt{|S|},
\end{align}
for some constant $C_7$.
\end{proof}

\begin{proof}[Proof of Lemma~\ref{lemma:gaussianmeasureLBforC}]

Let $C_3>\log(2)$ be large enough (to be fixed later, depending on $c_2$) and let $C_1$ be the constant that is guaranteed to exist by Lemma~\ref{lemma:matrixconcentrationestimateforS} (for the particular choice of $C_3$).

    Let $\Sigma$ denote the set of isomorphism classes of irreducible representations of $G$, and let $d_\pi$ denote the degree of $\pi\in\Sigma$. Pick a representative $\lambda_\pi\in\pi$ per isomorphism class. By the Peter--Weyl theorem (cf. proof of Lemma~7 in~\cite{Bandeira:Cayley24}), there exists a unitary matrix $U$ (depending only on the group) such that 
    \[
    \rho(g) = U\bigg[\bigoplus_{\pi\in\Sigma} (I_{d_\pi}\otimes \lambda_\pi(g))\bigg]U^*,
\qquad
g\in G.
\]
Thus,
\[
\PP\left\{ \left\|\sum_{g\in S}x_g\rho(g)\right\| \leq C_1\sqrt{|S|\log \left(e\frac{|G|}{|S|}\right)}  \right\}  = \PP
\bigcap_{\pi \in \Sigma} \left\{
\left\|\sum_{g\in S}x_g\lambda_{\pi}(g)\right\| \leq C_1\sqrt{|S|\log \left(e\frac{|G|}{|S|}\right)} \right\} 
.
    \]
Since this corresponds to the Gaussian measure of the intersection of centrally symmetric convex sets, the Gaussian Correlation Inequality~\cite{Royen:14,LatalaM:17} implies
\begin{align}
&\PP\left\{ \left\|\sum_{g\in S}x_g\rho(g)\right\| \leq C_1\sqrt{|S|\log \left(e\frac{|G|}{|S|}\right)} \right\} \nonumber\\
&\quad\geq \prod_{\pi \in \Sigma}\PP\left\{ \left\|\sum_{g\in S}x_g\lambda_{\pi}(g)\right\| \leq C_1\sqrt{|S|\log \left(e\frac{|G|}{|S|}\right)} \right\} \nonumber\\
&\quad= \prod_{\pi\in\Sigma} (1-p_\pi) = \exp\left( \sum_{\pi\in\Sigma} \log(1-p_\pi)\right), \label{eq:afterGCI}
\end{align}
where we define $p_\pi:=\PP\left\{ \left\|\sum_{g\in S}x_g\lambda_{\pi}(g)\right\| > 
C_1\sqrt{|S|\log \left(e\frac{|G|}{|S|}\right)}
\right\}$.
By Lemma~\ref{lemma:matrixconcentrationestimateforS} we have $p_\pi \leq \exp\left[-C_3  \log \left(e\frac{|G|}{|S|}\right)\right]\leq  \exp\left(-C_3\right)\leq \frac12$. Thus, we have
\begin{equation}
    \PP\left\{ \left\|\sum_{g\in S}x_g\rho(g)\right\| \leq C_1\sqrt{|S|\log \left(e\frac{|G|}{|S|}\right)}  \right\}  \geq \prod_{\pi\in\Sigma} (1-p_\pi) 
    \geq \exp\left( -2\sum_{\pi\in\Sigma} p_\pi\right), 
\end{equation}
where the last inequality used $p_\pi\leq \frac12$. Note that
\[
\frac2{|S|}\sum_{\pi\in\Sigma} p_\pi \leq \frac2{|S|}|\Sigma| \exp\left[-C_3  \log \left(e\frac{|G|}{|S|}\right)\right]\leq 2\frac{|G|}{|S|} \left(e\frac{|G|}{|S|}\right)^{-C_3} = 2e^{-C_3}\left(\frac{|S|}{|G|}\right)^{C_3-1}.
\]
For any $c_2>0$, there exists $C_3$ large enough that
\[
2e^{-C_3}\left(\frac{|S|}{|G|}\right)^{C_3-1}\leq c_2 \qquad\text{for all } |S|\leq |G|.
\]
Picking such $C_3$ yields
\begin{equation*}
    \PP\left\{ \left\|\sum_{g\in S}x_g\rho(g)\right\| \leq C_1\sqrt{|S|\log \left(e\frac{|G|}{|S|}\right)}  \right\}
    \geq \exp\left( -2\sum_{\pi\in\Sigma} p_\pi\right) 
    \geq \exp\left(-c_2|S|\right).
\end{equation*}
\end{proof}

\section{Rounds of Partial Coloring}\label{sec:roundspartialcoloring}

With Lemma~\ref{lemma:gaussianmeasureLBforC} established, the rest of the proof relies on the now-standard partial-coloring strategy dating back to the work of Gluskin~\cite{Gluskin1989,Giannopoulos1997,Bansal2022}.

\begin{lemma}[{\cite{Gluskin1989}; see~\cite{Bansal2022} for this formulation}]\label{lemma:gluskinpartialcoloring}
    Let $n\geq1$ be an integer. There is a small constant $\delta>0$ such that any symmetric convex body $K\subseteq\RR^n$ with $\PP(g\in K)\geq 2^{-\delta n}$, where $g\sim\mathcal N(0,I_n)$, contains a point $\gamma\in\{-1,0,1\}^n$ with at least $\delta n$ coordinates equal to $\pm1$.
\end{lemma}

\begin{proof}[Proof of Theorem~\ref{thm:main}]
    We adopt the nomenclature of partial colorings and coloring rounds: the goal is to assign a value $\pm1$, referred to as a color, to each entry of $\eps$. Let $S_t$ be the set of uncolored coordinates after round $t$, with $S_0=G$. At round $t$ we use Lemmata~\ref{lemma:gaussianmeasureLBforC} and \ref{lemma:gluskinpartialcoloring} (with $c_2=\delta \log 2$) for $S_t\subseteq G$ to obtain a partial coloring of $S_t$ in which at least $\delta|S_t|$ coordinates are colored. Thus, for all $t\geq 0$, $|S_{t+1}| \leq (1-\delta)|S_{t}|$, and since the $|S_{t}|$ are all integers, the process must terminate with $S_T=\emptyset$ for some $T$.

    By the triangle inequality, the resulting assignment satisfies 
    \begin{equation}\label{eq:partialcoloringtriangular}
\left\|\sum_{g\in G}\eps_g\rho(g)\right\| \leq \sum_{t=0}^{T-1} \left\|\sum_{g\in S_t\setminus S_{t+1}}\eps_g\rho(g)\right\| \leq \sum_{t=0}^{T-1} C_1\sqrt{|S_t|\log\left(e\frac{|G|}{|S_t|}\right)}.
    \end{equation}
    The derivative of $x\mapsto x\log(e/x)$ equals $\log(1/x)$, so $x\log(e/x)$ is increasing on $(0,1)$; hence $\sqrt{|S_t|\log\!\big(\tfrac{e|G|}{|S_t|}\big)}
\le\sqrt{(1-\delta)^t|G|\log\!\big(\tfrac{e}{(1-\delta)^t}\big)}$
since $|S_t|\le(1-\delta)^t|G|$.
Substituting into~\eqref{eq:partialcoloringtriangular}, we have
    \[
\left\|\sum_{g\in G}\eps_g\rho(g)\right\| \leq   C_1\sum_{t=0}^{T-1} \sqrt{(1-\delta)^t|G|\log\frac{e}{(1-\delta)^t}} = C_1\sqrt{|G|} \sum_{t=0}^{T-1} \sqrt{(1-\delta)^t\log\frac{e}{(1-\delta)^t}}.
    \]
The proof is completed by noting that, since $\delta>0$, there is a constant $C_\delta$ such that $\sum_{t=0}^{T-1} \sqrt{(1-\delta)^t\log\frac{e}{(1-\delta)^t}}\leq \sum_{t=0}^{\infty} \sqrt{(1-\delta)^t\log\frac{e}{(1-\delta)^t}}\leq C_\delta$.
\end{proof}




\section*{Acknowledgments and use of AI}
We used modern AI tools in this work, primarily ChatGPT Pro 5.5 and Claude Opus (versions 4.7 and 4.8). 
One of the key arguments---the one that broke the
$\sqrt{\log(|G|)}$ bottleneck, reducing it to a $\sqrt{\log\log(|G|)}$ factor
that we subsequently removed---was found by ChatGPT Pro 5.5. That argument no
longer appears in the paper as it was subsumed by the later improvement, but
in our view it was the decisive step in this work. The same AI tools were also used to organize partial results, run numerical experiments, and efficiently test ideas and approaches.
For instance, once we understood how low- and high-dimensional representations
could be handled simultaneously in Lemma~\ref{lemma:gaussianmeasureLBforC}, we
asked ChatGPT Pro 5.5 to check whether the two regimes together covered the full
range or left a gap ``in the middle''; in the argument as written, this split appears only at the very end, but in
earlier versions it occupied a more central place.

The paper was written entirely by the authors. AI was used to scan the manuscript
for typos and inconsistencies in style and notation, and to obtain general suggestions, some of which we adopted.
All errors are ours.

\bibliographystyle{alpha}   
\bibliography{references}

\end{document}